\documentclass[11pt]{article}
\usepackage[dvips]{graphicx}
\usepackage[latin1]{inputenc}
\usepackage{amsmath,amssymb,amsthm,mathrsfs,amsfonts}

\newtheorem{theorem}{Theorem}[section]
\newtheorem{lemma}[theorem]{Lemma}
\newtheorem{proposition}[theorem]{Proposition}

\title{On chromatic indices of finite affine spaces}

\author{Gabriela Araujo-Pardo \footnotemark[2] \and
Gy{\" o}rgy Kiss \footnotemark[3] \and
Christian Rubio-Montiel \footnotemark[4] \and
Adri{\' a}n V{\' a}zquez-{\' A}vila \footnotemark[5]}

\begin{document}

\maketitle

\def\thefootnote{\fnsymbol{footnote}}
\footnotetext[2]{Instituto de Matem{\' a}ticas, Universidad Nacional Aut{\' o}noma de M{\' e}xico, Ciudad Universitaria, 04510, Mexico City, Mexico, {\tt garaujo@matem.unam.mx}.}
\footnotetext[3]{Department of Geometry and MTA-ELTE Geometric and Algebraic Combinatorics Research Group, E{\" o}tv{\" o}s Lor{\' a}nd University, H-1117 Budapest, P{\' a}zm{\' a}ny s. 1/c, Hungary, and FAMNIT, University of Primorska, 6000 Koper, Glagolja\v ska 8, Slovenia, {\tt kissgy@cs.elte.hu}.}
\footnotetext[4]{Department of Algebra, Comenius University, Mlynska dolina, 84248, Bratislava, Slovakia, and Divisi{\' o}n de Matem{\' a}ticas e Ingenier{\' i}a, FES Acatl{\' a}n, Universidad Nacional Aut{\'o}noma de M{\' e}xico, 53150 Naucalpan, Mexico, {\tt christian.rubio@fmph.uniba.sk}.} 
\footnotetext[5]{Subdirecci{\' o}n de Ingenier{\' i}a y Posgrado, Universidad Aeron{\' a}utica en Quer{\' e}taro, 
Parque Aeroespacial Quer{\' e}taro, 76270, Quer{\' e}taro, M{\' e}xico, {\tt adrian.vazquez@unaq.edu.mx}.}

\begin{abstract}
The pseudoachromatic index of the finite affine space $\mathrm{AG}(n,q),$ denoted by $\psi'(\mathrm{AG}(n,q)),$ 
is the the maximum number of colors in any complete line-coloring of $\mathrm{AG}(n,q).$ 
When the coloring is also proper, the maximum number of colors is called the achromatic index of $\mathrm{AG}(n,q).$
We prove that if
$n$ is even then $\psi'(\mathrm{AG}(n,q))\sim q^{1.5n-1}$; while when $n$ is odd the value is bounded by $q^{1.5(n-1)}<\psi'(\mathrm{AG}(n,q))<q^{1.5n-1}$. 
Moreover, we prove that the achromatic index of $\mathrm{AG}(n,q)$ is $q^{1.5n-1}$ for even $n,$ and we provides the exact values of 
both indices in the planar case.

\end{abstract}


\section{Introduction}

This paper is motivated by the well-known combinatorial conjecture about 
colorings of finite linear spaces formulated by Erd{\H o}s, Faber and Lov{\' a}sz in 1972. 
As a starting point, we briefly introduce some definitions and give the conjecture. Let $\mathbf{S}$ be a finite linear space.
A \emph{line-coloring} of $\mathbf{S}$ with $k$ colors is a surjective function $\varsigma $ from the lines of $\mathbf{S}$ to the set of colors $[k]=\{1,\dots,k\}$. For short, a line-coloring  with $k$ colors is called \emph{$k$-coloring}. If $\varsigma \colon \mathbf{S}\rightarrow [k]$ is a $k$-coloring and $i\in [k]$ then the subset of lines $\varsigma ^{-1}(i)$ is called the $i$-th \emph{color class} of $\varsigma .$ A $k$-coloring of $\mathbf{S}$ is \emph{proper} if any two lines from the same color class have no point in common. The \emph{chromatic index} $\chi'(\mathbf{S})$ of $\mathbf{S}$ is the smallest $k$ for which there exists a proper $k$-coloring of $\mathbf{S}.$ The \emph{Erd{\H o}s-Faber-Lov{\' a}sz conjecture} (1972) states that if a finite linear space $\mathbf{S}$ contains $v$ points then $\chi'(\mathbf{S})\leq v,$ see \cite{MR0409246,MR602413}.

Many papers deal with the conjecture for particular classes of linear spaces. For instance, if each line of $\mathbf{S}$ has the same number $\kappa$ of points then $\mathbf{S}$ is called a \emph{block design} or a \emph{$(v,\kappa)$-design}. The conjecture is still open for designs even when $\kappa=3,$ 
however, it was proved for finite projective spaces by Beutelspacher, Jungnickel and Vanstone \cite{bjv}. 
It is not hard to see that the
conjecture is also true for the $n$-dimensional finite affine space of order $q$, denoted by $\mathrm{AG}(n,q),$ which has $q^n$ points. In fact, 
\begin{equation}\label{eq:igualdad}
\chi'(\mathrm{AG}(n,q))= \frac{q^{n}-1}{q-1}.
\end{equation}
Related results proved by some authors of this paper can be found in \cite{ARV,AV}.

A natural question is to determine similar, but slightly different color parameters in finite linear spaces. 
A $k$-coloring of $\mathbf{S}$ is \emph{complete} if for each pair of different colors $i$ and $j$ there exist two intersecting lines of $\mathbf{S}$,
such that one of them belongs to the $i$-th and the other one to the $j$-th color class. Observe that any proper coloring of $\mathbf{S}$ 
with $\chi'(\mathbf{S})$ colors is a complete coloring. The \emph{pseudoachromatic index} $\psi'(\mathbf{S})$ of $\mathbf{S}$ 
is the largest $k$ such that there exists a complete $k$-coloring (not necessarily proper) of $\mathbf{S}$. 
When the $k$-coloring is required to be complete and proper, the parameter is called the \emph{achromatic index} 
and it is denoted by $\alpha'(\mathbf{S})$. Therefore, we have that 
\begin{equation}\label{iquality:chi_alpha_psi}
\chi'(\mathbf{S}) \leq \alpha'(\mathbf{S}) \leq \psi'(\mathbf{S}).
\end{equation}
Several authors studied the pseudoachromatic index, 
see \cite{MR3249588,AMRS17,MR2778722,AR17,MR543176,MR0256930,MR0272662,MR989126}. Moreover, in \cite{AKRV,MR695809,MR1178507} the achromatic indices of some block designs were also estimated.

The objective of this paper is to study the pseudoachromatic and achromatic indices of finite affine spaces. 
Let $V_{n}$ be an $n$-dimensional vector space over the finite field of $q$ elements $\mathrm{GF}(q).$
The \emph{$n$-dimensional Desarguesian finite affine space} $\mathrm{AG}(n,q)$ is the geometry whose $k$-dimensional 
affine subspaces for $k=0,1,\dots ,n-1$ are the translates of the $k$-dimensional linear subspaces of $V_{n}.$ 
Thus any $k$-dimensional affine subspace can be given as:
$$\Sigma _k=L_k+\mathbf{v} = \{ \mathbf{x}+\mathbf{v}\, \colon \, \mathbf{x}\in  L_k\} $$
where $L_k$ is a $k$-dimensional linear subspace and $\mathbf{x}$ is a fixed element of $V_n.$
Subspaces of dimensions $0,1,2$ and $n-1$ are called \emph{points, lines, planes} and \emph{hyperplanes}, respectively.
Two affine subspaces $\Sigma _i$ and $\Sigma _j$ are said to be \emph{parallel}, if there exists 
$\mathbf{v}\in V_{n}$ for which $\Sigma _i + \mathbf{v} \subseteq \Sigma _j$ or $\Sigma _j + \mathbf{v} \subseteq \Sigma _i.$ 
In particular, two lines are parallel if and only if they are translates of the same $1$-dimensional linear subspace of
$V_n.$ Affine spaces are closely connected to projective spaces. Let $V_{n+1}$ be an $(n+1)$-dimensional vector space over $\mathrm{GF}(q).$
The \emph{$n$-dimensional Desarguesian finite projective space}, $\mathrm{PG}(n,q),$ is the geometry
whose $k$-dimensional subspaces for $k=0,1,\ldots ,n$ are the $(k+1)$-dimensional subspaces of $V_{n+1}.$
Let $\mathcal{H}_{\infty }$ be a fixed hyperplane in $\mathrm{PG}(n,q).$ If we delete all points of $\mathcal{H}_{\infty }$
from $\mathrm{PG}(n,q)$ then we obtian $\mathrm{AG}(n,q).$ The deleted points can be identified with the parallel classes of 
lines in $\mathrm{AG}(n,q).$ These points are called \emph{points at infinity} and we often consider the affine space as 
$\mathrm{AG}(n,q)=\mathrm{PG}(n,q)\setminus \mathcal{H}_{\infty }.$
For the detailed description of these spaces we refer to \cite{jwph3}. 

\smallskip
The results are organized as follows. In Section \ref{section2} the following upper bound is proved:
\begin{theorem}\label{thm:uno}
Let $v=q^n$ denote the number of points of the finite affine space $\mathrm{AG}(n,q).$ Then
\[\psi'(\mathrm{AG}(n,q))\leq\frac{\sqrt{v}(v-1)}{q-1}-\mathrm{\Theta }(q\sqrt{v}/2).\]
\end{theorem}
In Section \ref{section3} lower bounds for pseudoachromatic and achromatic indices of $\mathrm{AG}(n,q)$ are presented.
The main results are the following. 
\begin{theorem}\label{thm:dos}
Let $v=q^n$ denote the number of points of $\mathrm{AG}(n,q)$. 
\begin{itemize}
\item If $n$ is even:
\[ \frac{1}{2}\cdot\frac{\sqrt{v}(v-1)}{q-1}-\mathrm{\Theta }(\sqrt{v}/2) \leq \psi'(\mathrm{AG}(n,q)).\]
\item If $n$ is odd: 
\[ \frac{1}{\sqrt{q}}\cdot\frac{\sqrt{v}(v-1)}{q-1}-\mathrm{\Theta }(v\sqrt{v/q^5})\leq  \psi'(\mathrm{AG}(n,q)).\]
\end{itemize}
\end{theorem} 

\begin{theorem}\label{thm:tres}
Let $v=q^n$ denote the number of points of $\mathrm{AG}(n,q).$ If $n$ is even:
\[ \frac{1}{3}\cdot\frac{\sqrt{v}(v-1)}{q-1}+\mathrm{\Theta }(v/q) \leq \alpha'(\mathrm{AG}(n,q)).\]
\end{theorem}
\noindent
Note that when $n$ is even Theorems \ref{thm:uno} and \ref{thm:dos} show that $\psi'(\mathrm{AG}(n,q))$ grows asymptotically 
as $\mathrm{\Theta }(v^{1.5}/q),$ while Theorems \ref{thm:dos} and \ref{thm:tres} show that $\alpha '(\mathrm{AG}(n,q))$ grows asymptotically 
as $\mathrm{\Theta }(v^{1.5}/q).$

Finally, in Section \ref{section4} we determine the exact values of pseudoachromatic and
achromatic indices of arbitrary (not necessarily Desarguesian) finite affine planes and we improve the previous lower bounds in dimension $3.$


\section{Upper bounds}\label{section2}

In this section upper bounds for the pseudoachromatic index of $\mathrm{AG}(n,q)$ are presented when $n>2.$ 
The following lemma is pivotal in the proof.

\begin{lemma}
\label{intersect}
Let $n>2$ be an integer and $\cal L$ be a set of $s$ lines in $\mathrm{AG}(n,q).$ Then the number of lines 
in $\mathrm{AG}(n,q)$ intersecting at least one element of $\cal L$ is at most
$$q^2\left( s\frac{q^{n-1}-1}{q-1}-(s-1)\right) .$$ 
\end{lemma}
\begin{proof}
Recall that in $\mathrm{AG}(n,q)$ there exists a unique line joining any pair of points, and each line has exactly $q$ points. 
Hence there are $\frac{q^n-1}{q-1}$ lines through each point. Thus 
there are $$q\left( \frac{q^n-1}{q-1}-1\right) =q^2\left( \frac{q^{n-1}-1}{q-1}\right)$$ 
lines intersecting any fixed line.  
We claim that if $\ell _1$ and $\ell _2$ are different lines then the number of lines
intersecting both $\ell _1$ and $\ell _2$ is at least $q^2.$ If $\ell _1\cap \ell _2=\emptyset $ then the 
$q^2$ lines joining a point of $\ell _1$ and a point of $\ell _2$ intersect both $\ell _1$ and $\ell _2,$ 
while, if $\ell _1\cap \ell _2=\{ P\} $ then the other $(q^{n-1}+q^{n-2}+\dots +1)-2>q^2$ lines through $P$ intersect both $\ell _1$ and $\ell _2.$ 
Consequently, the number of lines
intersecting at least one element of $\cal L$ is at most
$$sq^2\left( \frac{q^{n-1}-1}{q-1}\right) -(s-1)q^2.$$

Notice that the previous inequality is tight, since if $\cal L$ consists of $s$ parallel 
lines in a plane then there are exactly $q^2\left( s\frac{q^{n-1}-1}{q-1}-(s-1)\right) $ 
lines intersecting at least one element of $\cal L$. 
\end{proof}

\begin{lemma}
Let $n>2$ be an integer.
Then the colorings of the finite affine space $\mathrm{AG}(n,q)$ satisfy the inequality 
\begin{equation}\label{upperest}
\psi' (\mathrm{AG}(n,q))\leq \frac{\sqrt{4q^{n}(q^n-1)(q^{n}-q^2)+(q^2+1)^2(q-1)^2}}{2(q-1)} +\frac{q^2+1}{2}.
\end{equation}
\end{lemma}
\begin{proof}
Consider a complete coloring which contains $\psi' (\mathrm{AG}(n,q))$ color
classes. Then the number of lines in the smallest color class is at most 
$$s=\frac{q^{n-1}(q^n-1)}{(q-1)\psi' (\mathrm{AG}(n,q))}.$$ 
Each of the other $\psi' (\mathrm{AG}(n,q))-1$
color classes must contain at least one line which intersects a line of the
smallest color class. Hence, by Lemma \ref{intersect}, we obtain 
$$\psi' (\mathrm{AG}(n,q))-1\leq q^2\left( s\frac{q^{n-1}-1}{q-1}-(s-1)\right) .$$
Multiplying it by $\psi' (\mathrm{AG}(n,q)),$ we get a quadratic inequality on 
$\psi' (\mathrm{AG}(n,q)),$ whose solution yields the statement of the theorem.
\end{proof}

We can now prove our first main theorem.

\begin{proof}[Proof of Theorem \ref{thm:uno}]
In the case $n>2$ elementary calculation yields
$$
\begin{aligned}
4q^{n}(q^n-1)(q^{n}-q^2)+(q^2+1)^2(q-1)^2  = & \left( 2q^{\frac{n}{2}}(q^{n}-1)
-q^{\frac{n}{2}}(q^2-1)\right) ^2 \\
 & - q^n(q^2-1)^2 +(q^2+1)^2(q-1)^2 \\
 < & \left( 2q^{\frac{n}{2}}(q^{n}-1)
-q^{\frac{n}{2}}(q^2-1)\right) ^2,
\end{aligned}
$$
because $n>2$ implies that $q^n(q^2-1)^2 >(q^2+1)^2(q-1)^2.$
Thus we can estimate the radical expression in Equation (\ref{upperest}) and we obtain 
\[\psi' (\mathrm{AG}(n,q))\leq q^{\frac{n}{2}}\left( \frac{q^{n}-1}{q-1}\right) -q^{\frac{n}{2}}\left ( \frac{q+1}{2}\right )+\frac{q^2+1}{2}, \]
which proves the theorem for $n>2.$ For $n=2$ the statement is clear.
\end{proof}


\section{Lower bounds}\label{section3}

In this section we prove a lower bound on the pseudoachromatic index of $\mathrm{AG}(n,q)$. 
To achieve this we present complete colorings of $\mathrm{AG}(n,q)$. 
The constructions depend on the parity of the space dimension.
First, we prove some geometric properties of affine and projective spaces.

\begin{proposition}
\label{dim}
Let $n>1$ be an integer, $\Pi _1$ and $\Pi _2$ be subspaces in 
$\mathrm{PG}(n,q)=\mathrm{AG}(n,q)\cup {\mathcal H}_{\infty }.$
Let $d_i$ denote the dimension of $\Pi _i$ for $i=1,2.$ 
Suppose that $\Pi _1\cap\Pi _2\cap {\mathcal H}_{\infty }$ is an
$m$-dimensional subspace and $d_1+d_2=n+1+m.$
Then $\Pi _1\cap\Pi _2\cap \mathrm{AG}(n,q)$ is an $(m+1)$-dimensional
subspace in $\mathrm{AG}(n,q).$ 

In particular, $\Pi _1\cap\Pi _2$ is a single point in $\mathrm{AG}(n,q)$ when $\Pi _1\cap\Pi _2\cap {\mathcal H}_{\infty }=\emptyset $ 
and $d_1+d_2=n.$

\end{proposition}
 
\begin{proof}
Since $\Pi _1\cap\Pi _2\cap {\mathcal H}_{\infty }$  is an
$m$-dimensional subspace, the subspace $\Pi _1\cap\Pi _2$ has dimension at most $m+1.$ 
On the other hand, the dimension formula yields
$$\mathrm{dim}(\Pi _1\cap \Pi _2)=\mathrm{dim}\, \Pi _1+\mathrm{dim}\, \Pi _v-\mathrm{dim}\langle \Pi _1,\Pi _2\rangle
\geq d_1+d_2-n=m+1,$$ 
therefore $\Pi _1\cap\Pi _2\cap \mathrm{AG}(n,q)$ is an $(m+1)$-dimensional
subspace in $\mathrm{AG}(n,q).$ 

If $m=-1$, $\Pi _1\cap\Pi _2\cap {\mathcal H}_{\infty }=\emptyset ,$ 
but the subspace $\Pi _1\cap\Pi _2$ has dimension 0 in
$\mathrm{PG}(n,q).$ Hence, it is a single point and this point is 
not in ${\mathcal H}_{\infty }$ so it is in $\mathrm{AG}(n,q).$
\end{proof}

In the following proposition we present a partition of the points of 
$\mathrm{PG}(2k,q)$ that we will call \emph{good partition}
in the rest of the paper.

\begin{proposition}
\label{assign}
Let $k\geq 1$ be an integer and $Q\in \mathrm{PG}(2k,q)$ be an arbitrary point.
The points of $\mathrm{PG}(2k,q)\setminus \{ Q\} $ can be divided into
into two subsets, say ${\mathcal A}$ and ${\mathcal B},$
and one can assign a subspace $S(P)$ to each point 
$P\in {\mathcal A}\cup {\mathcal B}$, such that the following holds true.
\begin{itemize}
\item
$P\in S(P)$ for all points,
\item  
$|{\mathcal A}|= q^2\left( \frac{q^{2k}-1}{q^2-1}\right)$  
and, if $A\in {\mathcal A}$ then 
$S(A)$ is a $k$-dimensional
subspace, 
\item
$|{\mathcal B}|= q\left (\frac{q^{2k}-1}{q^2-1}\right )$ and,
if $B\in {\mathcal B}$ then  
$S(B)$  is a $(k-1)$-dimensional
subspace,
\item
$S(A)\cap S(B)=\emptyset $ for all $A\in {\mathcal A}$ and $B\in {\mathcal B}.$
\end{itemize}
\end{proposition}

\begin{proof}
We prove by induction on $k.$ If $k=1$ then let $\{ \ell _0,\ell _1,\dots ,\ell _q\} $
be the set of lines through $Q.$ Let ${\mathcal A}$ and ${\mathcal B}$ consist of points $\mathrm{PG}(2,q)\setminus \{ \ell _0\} $ 
and $ \ell _0\setminus \{ Q\},$ respectively. If $A\in {\mathcal A}$ then let 
$S(A)$ be the line $AQ,$ if $B\in {\mathcal B}$ then let 
$S(B)$ be the point $B.$ These sets clearly fulfill the 
prescribed conditions, so $\mathrm{PG}(2,q)$ admits a good partition.

Now, let us suppose that $\mathrm{PG}(2k,q)$ admits a good partition. 
In $\mathrm{PG}(2k+2,q)$ take a $2k$-dimensional
subspace $\Pi $ which contains the point $Q.$ Then $\Pi $ is isomorphic to $\mathrm{PG}(2k,q),$ hence it has a good 
partition $\{ Q\} \cup {\mathcal A}'\cup {\mathcal B}'$ with assigned subspaces $S'(P).$ 
Let $H_0,H_1,\ldots ,H_q$ be the pencil of hyperplanes in $\mathrm{PG}(2k+2,q)$ with carrier $\Pi .$
Let ${\mathcal B}={\mathcal B}'\cup (H_0\setminus \Pi )$ and 
${\mathcal A}= \mathrm{PG}(2k+2,q)\setminus ({\mathcal B}\cup \{ Q\} ).$ 
Notice that ${\mathcal A}'$ and ${\mathcal B}'$ have the required cardinalities, because
$$
\begin{aligned}
|{\mathcal A}'| & =  \frac{q^{2k+3}-1}{q-1}-(|{\mathcal B}|+1)=(q+1)\frac{q^{2k+3}-1}{q^2-1}-q\left( \frac{q^{2k+2}-1}{q^2-1}\right)-1 \\
 & =q^2\left (\frac{q^{2k+2}-1}{q^2-1}\right), \\
|{\mathcal B}'| & =|{\mathcal B}|+ |(H_0\setminus \Pi |=q\left( \frac{q^{2k}-1}{q^2-1}\right) + q^{2k+1}=q\left (\frac{q^{2k+2}-1}{q^2-1}\right ).
\end{aligned}
$$

We assign the subspaces in the following way.
If $A\in {\mathcal A}'$ then let
$S(A)$ be the $(k+1)$-dimensional subspace $\langle S'(A), P\rangle $ where 
$P\in \cup_{i=1}^qH_i$ is an arbitrary point, whereas,
if $A\in (\cup_{i=1}^qH_i)\setminus \Pi $ then let 
$S(A)$ be the $(k+1)$-dimensional subspace $\langle A, S'(P)\rangle $ where 
$P\in {\mathcal A}'$ is an arbitrary point. In both cases $S(A)\subset \cup_{i=1}^qH_i$
for all $A\in {\mathcal A}.$  
Similarly, if $B\in {\mathcal B}'$ then let
$S(B)$ be the $k$-dimensional subspace $\langle S'(B), P\rangle $ where 
$P\in H_0$ is an arbitrary point, whereas, 
if $B\in H_0\setminus \Pi $ then let 
$S(B)$ be the $k$-dimensional subspace $\langle B, S'(P)\rangle $ where 
$P\in {\mathcal B}'$ is an arbitrary point. Also here, in both cases, $S(B)\subset H_0$
for all $B\in {\mathcal B}.$  
Moreover, the assigned subspaces satisfy the intersection condition because
if $A\in {\mathcal A}$ and $B\in {\mathcal B}$ are arbitrary points then
$$S(A)\cap S(B) =(S(A)\cap (\cup_{i=1}^qH_i))\cap (S(B)\cap H_0) = S'(A)\cap S'(B)\cap \Pi=\emptyset .$$

Hence $\mathrm{PG}(2k+1,q)$ also admits a good partition, the statement is proved.
\end{proof}

The next theorem proves Theorem \ref{thm:dos} for even dimensional finite affine spaces. 
Notice that the lower bound depends on the parity of $q$, but its magnitude is
$\frac{\sqrt{v}(v-1)}{2(q-1)}$ in both cases, where $v=q^n$.

\begin{theorem}\label{even-part}
If $k>1$ then the colorings of the even dimensional affine space, $\mathrm{AG}(2k,q),$ satisfy
the inequalities 
$$\psi' (\mathrm{AG}(2k,q))\geq \left\{  
\begin{array}{ll}
\frac{q^k(q^{2k}-1)}{2(q-1)}, & \text{if q is odd}, \\  
\frac{q^k(q^{2k}-q)}{2(q-1)} +1, & \text{if q is even.}
\end{array}
\right. $$
\end{theorem}

\begin{proof}
Consider the projective closure of the affine space,
$\mathrm{PG}(2k,q) = \mathrm{AG}(2k,q)\cup \mathcal{H}_{\infty }.$
The parallel classes of affine lines correspond to the points of
${\mathcal H}_{\infty }.$ The hyperplane at infinity is isomorphic to
the projective space $\mathrm{PG}(2k-1,q),$ hence it has a $(k-1)$-spread
${\cal S}=\{ S^1,S^2,\dots ,S^{q^k+1}\} $. The elements of $\cal S$ are
pairwise disjoint $(k-1)$-dimensional subspaces (see \cite[Theorem 4.1]{jwph3}). 
Let $\{ P^i_1,P^i_2,\dots ,P^i_{(q^k-1)(q-1)}\} $ be the set of points of $S^i$ for $i=1,2,\dots , q^k+1$.

We define a pairing on the set of points of ${\mathcal H}_{\infty }$ which depends on the parity of $q.$ 
On the one hand, if $q$ is odd then let $(P^i_j,P^{i+1}_j)$ be the pairs for $i=1,3,5,\dots ,q^k$ and $j= 1,2,\dots , \frac{q^k-1}{q-1}$. 
On the other hand, if $q$ is even then ${\mathcal H}_{\infty }$ has an odd number of points, thus we give the paring on the 
set of points ${\mathcal H}_{\infty }\setminus \{ P^1_1\}$: let $(P^i_j,P^{i+1}_j)$ be the pairs for $i=4,6,\dots ,q^k$ and $j= 1,2,\dots , \frac{q^k-1}{q-1}$, and let $(P^1_j,P^{2}_j),$ $(P^2_{j+1},P^{3}_{j+1}),$  $(P^1_{j+1},P^{3}_{j})$
and $(P^2_1,P^3_1)$ be the pairs for $i=1,2,3$ and $j= 2,4,6,\dots , \frac{q^k-1}{q-1}-1.$
 
For $P\in {\mathcal H}_{\infty }$ we denote by $S(P)$ the unique element of $\cal S$ that contains $P.$ 
Consider the set of $k$-dimensional subspaces of 
$\mathrm{PG}(2k,q)$ intersecting ${\mathcal H}_{\infty }$ in $S(P).$ The affine parts of
these subspaces determine a set of $q^k$ parallel $k$-dimensional subspaces of 
$\mathrm{AG}(2k,q)$, denoted by $A(P)=\{ \Pi _{P,1},\Pi _{P,2},\dots ,\Pi _{P,q^k}\}. $ 

Let $(U,V)$ be any pair of points. Then, by defintion,
$S(U)\neq S(V).$ Let the color class $C_{U,V,i}$ contain the 
lines joining either $U$ and a point from $\Pi _{U,i},$ or $V$ and a point 
from $\Pi _{V,i}$, for $i=1,2,\dots ,q^k$. Clearly, $(U,V)$ defines $q^k$ color classes, each one consists 
of the parallel lines of one subspace in $A(U)$ and the parallel lines of one subspace in $A(V)$.
Finally, if $q$ is even, then let the color class $C_1$ consist of all lines 
of $\mathrm{AG}(2k,q)$ whose point at infinity is $P^1_1.$

We divided the points of ${\mathcal H}_{\infty }$ into $\frac{q^{2k}-1}{2(q-1)}$ 
pairs if $q$ is odd, and into $\frac{q^{2k}-q}{2(q-1)}$ 
pairs if $q$ is even. Consequently, the number of color classes is equal to $\frac{q^{2k}-1}{2(q-1)}q^k$ when $q$ is odd, and
it is equal to $\frac{q^{2k}-q}{2(q-1)}q^k+1$ when $q$ is even.

Now, we show that the coloring is 
complete. The class $C_1$ obviously intersects any other class.
Let $C_{U,V,i}$ and $C_{W,Z,j}$ be two color classes.
Then $S(U)$ and $S(V)$ are distinct elements of the spread $\cal S$ and $S(W)$ is also an element of $\cal S$. 
Hence we may assume, without loss of generality, that 
$S(U)\cap S(W)=\emptyset .$ As $\mathrm{dim}\, (S(U)\cup \Pi _{U,i})=\mathrm{dim}\, (S(W)\cup \Pi _{W,j})=k$
in $\mathrm{PG}(2k,q),$ by
Proposition \ref{dim}, we have that $\Pi _{U,i}\cap \Pi _{W,j}$ consists of a single point in 
$\mathrm{AG}(2k,q).$
Notice that the coloring is not proper, because the same argument shows that 
$\Pi _{U,i}\cap \Pi _{V,i}$ is also a single point in $\mathrm{AG}(2k,q).$
\end{proof}

For odd dimensional finite affine spaces we have a slightly weaker estimate. In this case, 
the magnitude of the lower bound is $\frac{1}{\sqrt{q}}\cdot\frac{\sqrt{v}(v-1)}{q-1},$ where $v=q^n$.

\begin{theorem}\label{odd-part}
If $k\geq 1$ then the colorings of the odd dimensional affine space, $\mathrm{AG}(2k+1,q),$ 
satisfy the inequality 
$$q^{k+2}\left(\frac{q^{2k}-1}{q^2-1}\right)+1   \leq  \psi' (\mathrm{AG}(2k+1,q)).$$
\end{theorem}

\begin{proof}
Consider the projective closure of the affine space
$\mathrm{PG}(2k+1,q)=\mathrm{AG}(2k+1,q)\cup {\mathcal H}_{\infty }.$
Here, the parallel classes of affine lines correspond to the points of
${\mathcal H}_{\infty },$ and the hyperplane at infinity is isomorphic to
the projective space $\mathrm{PG}(2k,q).$ 

By Proposition \ref{assign}, ${\mathcal H}_{\infty }$ admits a good partition.
So we can divide the points of ${\mathcal H}_{\infty }$ into three disjoint classes where
${\cal A}=\{ P_1,P_2, \dots ,P_{t}\} $ and ${\cal B}=\{ R_1,R_2, \dots ,R_{s}\} $  are two sets of 
$t=q^2\left( \frac{q^{2k}-1}{q^2-1}\right) $ and
$s=q\left( \frac{q^{2k}-1}{q^2-1}\right) $ points, respectively, and 
${\cal C}=\{ Q\} $ is a set containing a single point.
We can also assign a subspace $S(U)$ to each 
point $U\in {\mathcal A}\cup {\mathcal B}$ such that 
$S(P_i)\subset {\mathcal H}_{\infty }$ is a $k$-dimensional
subspace if $P_i\in {\mathcal A}$ and 
$S(R_j)\subset {\mathcal H}_{\infty }$ is a $(k-1)$-dimensional
subspace if $R_j\in {\mathcal B},$ furthermore $S(P_i)\cap S(R_j)=\emptyset $ for all $i$ and $j.$

Consider the $(k+1)$-dimensional subspaces of 
$\mathrm{PG}(2k+1,q)$ that intersect ${\mathcal H}_{\infty }$ in $S(P_i).$ 
The affine parts of
these subspaces form a set of $q^k$ parallel $(k+1)$-dimensional subspaces of 
$\mathrm{AG}(2k+1,q).$  Let 
$A(P_i)=\{ \Pi _{P_i,1},\Pi _{P_i,2},\dots ,\Pi _{P_i,q^k}\} $ denote this set.
Similarly, consider the $k$-dimensional subspaces of 
$\mathrm{PG}(2k+1,q)$ intersecting ${\mathcal H}_{\infty }$ in $S(R_j).$ 
The affine parts of
these subspaces induce a set of $q^{k+1}$ parallel $k$-dimensional subspaces of 
$\mathrm{AG}(2k+1,q)$ denoted by 
$B(R_j)=\{ \Pi _{R_j,1},\Pi _{R_j,2},\dots ,\Pi _{R_j,q^{k+1}}\}. $

Now, we define the color classes. Let $C_1$ be the color class that contains 
all lines of $\mathrm{AG}(2k+1,q)$ whose point at infinity is $Q.$
Let the color class $C_{i,j,m}$ contain the 
lines joining either $P_{(j-1)q+i}$ and
a point from $\Pi _{P_{(j-1)q+i},m},$ or $R_j$ and a point from 
$\Pi _{R_j,(i-1)q^k+m}$ for $j=1,2,\dots ,s,$ $i=1,2,\dots ,q$ and $m=1,2,\dots ,q^k$.
Counting the number of color classes of type $C_{i,j,m},$ we obtain $s\cdot q\cdot q^k=q^{k+2}\left( \frac{q^{2k}-1}{q^2-1}\right) .$ 
Each color class consists of the parallel lines of 
one subspace in $A(P_{(j-1)q+i})$ and the parallel lines of one subspace in 
$B(R_j).$ Clearly, the total number of color classes is $1+q^{k+2}\left (\frac{q^{2k}-1}{q^2-1}\right).$
The color class $C_1$ contains 
$q^{2k}$ lines and each of the classes of type $C_{i,j,m}$ consists of 
$q^k+q^{k-1}$ lines.

To prove that the coloring is complete, notice that the class 
$C_1$ obviously intersects any other class. Let
$C_{i,j,m}$ and $C_{i',j',m'}$ be two color classes other than $C_1.$ Consider those elements of $A(P_{(j-1)q+i})$ and $B(R_j')$ whose lines are 
contained in $C_{i,j,m}$ and in $C_{i',j',m'}$, respectively. 
One of these subspaces is a $(k+1)$-dimensional
subspace, whereas the other one is a $k$-dimensional subspace in 
$\mathrm{PG}(2k+1,q),$ and they have no point in common in 
${\mathcal H}_{\infty }.$ Thus,
by Proposition \ref{dim}, their intersection is a single point in 
$\mathrm{AG}(2k+1,q).$

The coloring is not proper, because the same argument shows that 
$\Pi _{P_{(j-1)q+i},m} \cap  \Pi _{R_j,(i-1)q^k+m}$
is also a point in $\mathrm{AG}(2k+1,q),$ thus 
$C_{i,j,m}$ contains a pair of intersecting lines. 
\end{proof}

Now, we are ready to prove our second main theorem.

\begin{proof}[Proof of Theorem \ref{thm:dos}]
If $n$ is even then Theorem \ref{even-part} gives the result at once. If $n$ is odd
then $v=q^{2k+1},$ hence $\sqrt{v/q}=q^k.$ From the estimate of Theorem \ref{odd-part} we get 
$$
\begin{aligned}
q^{k+2}\left( \frac{q^{2k}-1}{q^2-1}\right )+1 & =  \frac{q^{3k+2}-q^{k+2}}{q^2-1}+1 \\
 & = \frac{(q+1)(q^{3k+1}-q^{k})}{q^2-1}-
\frac{q^{3k+1}+q^{k+2}-q^{k+1}-q^{k}}{q^2-1}+1 \\
 & = \frac{1}{\sqrt{q}}\frac{\sqrt{v}(v-1)}{q-1}-\frac{q^{3k+1}+q^{k+2}-q^{k+1}-q^{k}}{q^2-1}+1,
 \end{aligned} 
 $$
which proves the statement.
\end{proof}

Next, recall that a lower bound for the achromatic index require a proper and
complete line-coloring of $\mathrm{AG}(n,q)$. We consider only the even 
dimensional case.

\begin{theorem}\label{teo-achro}
Let $k>1$ and
$\epsilon =0,1$ or $2$, such that $q^k+1\equiv \epsilon$ (mod $3$). Then the achromatic index of the even dimensional finite 
affine space 
$\mathrm{AG}(2k,q)$ satisfies
the inequality
$$ 
\left( \frac{q^k+1-\epsilon }{3}(q^k+2)+\epsilon \right)\frac{q^k-1}{q-1}
\leq \alpha'(\mathrm{AG}(2k,q)).
$$
\end{theorem}

\begin{proof}
Again, consider the projective closure of the affine space 
$\mathrm{PG}(2k,q)=\mathrm{AG}(2k,q)\cup {\mathcal H}_{\infty }.$
The parallel classes of affine lines correspond to the points of
${\mathcal H}_{\infty },$ and the hyperplane at infinity is isomorphic to
$\mathrm{PG}(2k-1,q).$ 

Let $\mathcal{L}=\{ \ell _1,\ell _2,\dots ,\ell _{q^k+1}\} $ be a $(k-1)$-spread of 
${\mathcal H}_{\infty }.$ Consider the set of $k$-dimensional subspaces of 
$\mathrm{PG}(2k,q)$ intersecting ${\mathcal H}_{\infty }$ in $\ell_i.$ 
The affine parts of
these subspaces form a set of $q^k$ parallel $k$-dimensional subspaces in
$\mathrm{AG}(2k,q).$ Let 
$\mathcal{A}(\ell_i)=\{ \Pi _{\ell_i,1},\Pi _{\ell_i,2},\dots ,\Pi _{\ell_i,q^k}\} $ denote this set.  
By Proposition\ref{dim}, the intersection
$\Pi _{\ell_i,s}\cap \Pi _{\ell_j,t}$ is a single affine point
for all $i\neq j$ and $1\leq s,t\leq q^k.$

First, to any triple of $(k-1)$-dimensional 
subspaces $e,f,g\in \mathcal{L}$ we assign $q^k+2$ color classes as follows.
Take a fourth $(k-1)$-dimensional subspace $d\in \mathcal{L},$  
and, for $u=(q^k-1)/(q-1)$, denote the points of the $(k-1)$-dimensional 
subspaces $d,e,f$ and $g$ as 
$D_1,D_2,\dots ,D_u,$ 
$E_1,E_2,\dots ,E_u,$
$F_1,F_2,\dots ,F_u$ and 
$G_1,G_2,\dots ,G_u,$ respectively.         
For any triple $(D_i,e,g)$ there is a unique line through $D_i$ which intersects
the skew subspaces $e$ and $g.$ We can choose the numbering of the 
points $E_i$ and $G_i$, such that the line $E_iG_i$ intersects 
$d$ in $D_i$ for $i=1,2,\dots ,u;$
the numbering of the points $F_i$, such 
that the line $D_iF_{i+1}$ intersects $d$ and $g$ for $i=1,2,\dots ,u-1,$ 
and, finally choose the line $D_uF_1$ that intersects $d$ and $g.$ Notice that this construction implies that the line $D_iF_i$ 
does not intersect $g$ for $i=1,2,\dots ,u.$          
Let the points of $\Pi _{d,1}$ denote by $M_1,M_2,\dots ,M_{q^k}.$  
We can choose the numbering of the elements of
$\mathcal{A}(e),$ $\mathcal{A}(f)$ and $\mathcal{A}(g)$, such that
$\Pi _{e,i}\cap \Pi _{f,i}\cap \Pi _{g,i}=\{ M_i\}$ for
$i=1,2,\dots ,q^k.$

We define three types of color classes for $i=1,2,\dots ,u$ and $j=1,2,\dots ,q^k.$ 
Let $B^{i,0}_{e,f,g}$ and $B^{i,1}_{e,f,g}$ be the color 
classes that contain the affine parts of the lines $E_iM_j,$ and $F_iM_j,$ respectively.
Let $C^{i,j}_{e,f,g}$ be the color class that contains the affine parts of lines in $\Pi _{e,i}$ whose
point at infinity is $E_j,$ except the line $E_jM_i,$ the affine parts of lines in $\Pi _{f,i}$ whose point
at infinity is $F_j,$ except the line $F_jM_i,$ and the affine parts of lines in $\Pi _{g,i}$ 
whose point at infinity is $G_j.$ 
Hence each of $B^{i,0}_{e,f,g}$ and $B^{i,1}_{e,f,g}$ contains $q^k$ lines 
and $C^{i,j}_{e,f,g}$ contains $3q^{k-1}-2$ lines.

Notice that for each $i\in \{ 1,2,\dots ,u\} $, the union of the color classes 
$$\mathcal{K}^i_{e,f,g}=B^{i,0}_{e,f,g}\cup B^{i,1}_{e,f,g}\cup _{j=1}^{q^k}C^{i,j}_{e,f,g}$$
contains the affine parts of all lines of $\mathrm{PG}(2k,q)$ whose point at
infinity is $E_i,F_i$ or $G_i.$
Each of the two sets of lines whose affine parts belong to $B^{i,0}_{e,f,g}$ or $B^{i,1}_{e,f,g},$
naturally defines a $(k+1)$-dimensional subspace of $\mathrm{PG}(2k,q),$ we denote these subspaces by
$\Pi _{E_i}$ and $\Pi _{F_i},$ respectively.  

For $t=0,1,\dots , \lfloor (q^k-2-\epsilon )/3\rfloor $
let $e=\ell _{3t+1},$ $f=\ell _{3t+2},$ $g=\ell _{3t+3}$, $d=\ell _{3t+4}$, define $\ell_{q^k+2-\epsilon }$ as $\ell_1$, 
and make the $q^k+2$ color classes  
$B^{i,0}_{e,f,g},$  $B^{i,1}_{e,f,g}$ and $C^{i,j}_{e,f,g}.$ Finally, for each
point $P$ in the subspace $\ell _{q^k+1}$ if $\epsilon =1,$ or in $\ell _{q^k}$ if $\epsilon =2$, 
define a new color class $D^P$ which contains the affine parts of all lines whose
point at infinity is $P.$ 

Clearly, the coloring is proper and it contains, by definition, the required number of color classes. 
Now, we prove that it is complete.  
Notice that each color class of type $D^P$ obviously intersects any other color class. 
In relation to the other cases we have that: 
\begin{itemize}
\item
The color classes 
$B^{i,j}_{\ell _{3m+1},\ell _{3m+2},\ell _{3m+3}}$ 
and $B^{i',j'}_{\ell _{3m+1},\ell _{3m+2},\ell _{3m+3}}$ intersect, because both of 
them contain all affine points of the $k$-dimensional subspace 
$\Pi _{\ell _{3m+4},1}.$  
\item 
If $t\neq m$ then the color classes 
$B^{i,j}_{\ell _{3t+1},\ell _{3t+2},\ell _{3t+3}}$ and 
$B^{i',j'}_{\ell _{3m+1},\ell _{3m+2},\ell _{3m+3}}$ intersect, because the 
$(k-1)$-dimensional subspaces
$\ell _{3t+4}$ and $\ell _{3m+4}$ are skew in 
$\mathcal{H}_{\infty },$ hence the $2$-dimensional intersection 
of the $(k+1)$-dimensional subspaces $\Pi _{E_i}$ or $\Pi _{F_i},$ according as
$j=1$ or $2$, and $\Pi _{E_i'}$ or $\Pi _{F_i'},$ according as
$j'=1$ or $2,$ is not a subspace of $\mathcal{H}_{\infty }.$ Thus    
Proposition \ref{dim} implies that the intersection contains some
affine points.
\item
The color classes 
$B^{i,j}_{\ell _{3m+1},\ell _{3m+2},\ell _{3m+3}}$ 
and $C^{i',j'}_{\ell _{3t+1},\ell _{3t+2},\ell _{3t+3}}$  
intersect in both cases $m=t$ and $m\neq t,$ because the 
$(k-1)$-dimensional subspaces $\ell _{3m+4}$ and $\ell _{3t+3}$ are skew in  
$\mathcal{H}_{\infty }.$ Again, Proposition \ref{dim} implies that the 
intersection of the $k$-dimensional subspaces $\Pi _{\ell _{3m+4},1}$ (which
is a subspace of either the 
$(k+1)$-dimensional subspace $\Pi _{E_i}$ or $\Pi _{F_i},$ according as
$j=1$ or $2$) and $\Pi_ {\ell _{3m+3,i'}}$ is an affine point. 
\item 
If $t\neq m$ then each pair of color classes 
$C^{i,j}_{\ell _{3t+1},\ell _{3t+2},\ell _{3t+3}}$ and $C^{i',j'}_{\ell _{3m+1},\ell _{3m+2},\ell _{3m+3}},$ intersects since, as previously, the 
$(k-1)$-dimensional subspaces $\ell _{3t+3}$
and $\ell _{3m+3}$ are skew in $\mathcal{H}_{\infty },$ thus 
Proposition \ref{dim} implies that the point of intersection 
of the $k$-dimensional subspaces $\Pi _{\ell _{3t+3},i}$ and 
$\Pi _{\ell _{3m+3},i'}$ is in $\mathrm{AG}(2k,q).$
\item  
Finally, we prove that each pair of color classes 
$C^{i,j}_{\ell _{3t+1},\ell _{3t+2},\ell _{3t+3}}$ and $C^{i',j'}_{\ell _{3t+1},\ell _{3t+2},\ell _{3t+3}}$ intersects.
It is obvious when $i=i'.$ Suppose that $i\neq i'$, let 
$M_i=\Pi _{\ell _{3t+1},i} \cap \Pi _{\ell _{3t+2},i}\cap \Pi _{\ell _{3t+3},i}$
and $M_{i'}=\Pi _{\ell _{3t+1},i'} \cap \Pi _{\ell _{3t+2},i'}\cap \Pi _{\ell _{3t+3},i'}.$
Since the points $M_i$ and $M_{i'}$ are in $\Pi _{\ell_{3t+4},1}$, 
the line $M_iM_{i'}$ intersects $\mathcal{H}_{\infty }$ in $\ell _{3t+4}.$ Consider the point
$T=M_iM_{i'}\cap \ell _{3t+4}$ and the lines $E_jT$ and $F_jT.$ 
Clearly, at least one of these lines does not intersect 
$\ell _{3t+3}$, we may assume, without loss of generality, that the line 
$E_jT$ does not intersect the $(k-1)$-dimensional subspace
$\ell _{3t+3}$ in $\mathcal{H}_{\infty }.$

By Proposition \ref{dim}, there exist affine points
$N_i=\Pi _{\ell _{3t+1},i} \cap \Pi _{\ell _{3t+3},i'}$
and $N_{i'}=\Pi _{\ell _{3t+1},i'} \cap \Pi _{\ell _{3t+3},i}.$ 

Suppose that $N_i\in E_{j'}M_{i'}$ and $N_{i'}\in E_jM_i.$ Then
the intersection of the $(k-1)$-dimensional subspace $\ell _{3t+1}$ and 
the line $M_iM_{i'}$ is empty, hence these two subspaces generate a $(k+1)$-dimensional subspace $\Sigma _{k+1},$ 
which intersects $\mathcal{H}_{\infty }$ in a 
$k$-dimensional subspace $\Sigma _k.$ Obviously,  
$\Sigma _k$ also contains the points $E_j$ and $E_{j'}.$ Then
$\Sigma _k =\langle \ell _{3t+1},T\rangle$, and
$\Sigma _k\cap \ell _{3t+3}$ is a single point, say $U.$ 
As the lines $N_{i'}M_i$ and $N_{i}M_{i'}$ are in the $k$-dimensional subspaces $\Pi _{\ell _{3t+3},i}$ and $\Pi _{\ell _{3t+3},i'},$ respectively, 
there exist the points
$N_{i'}M_i\cap \ell _{3t+3}$ and
$N_{i}M_{i'}\cap \ell _{3t+3}$. Moreover, we have that 
$N_{i'}M_i\cap \ell _{3t+3}=N_{i}M_{i'}\cap \ell _{3t+3}=U.$
Hence the points $N_i,M_i,N_{i'}$ and $M_{i'}$ 
are contained in a $2$-dimensional subspace $\Sigma _2,$ and $\Sigma _2\cap \mathcal{H}_{\infty }$ 
contains the points $U,\, E_j,$ $E_{j'}$ and $T.$
Consequently, $\Sigma _2\cap \mathcal{H}_{\infty }$ is the line $E_jT$ and it contains the point $U,$
thus $E_jT$ intersects the subspace $\ell _{3t+3},$ contradiction.  

Thus $N_i\not\in E_{j'}M_{i'}$ or $N_{i'}\not\in E_jM_i.$ This implies that
$N_i$ or $N_{i'}$ is a common point of the color classes 
$C^{i,j}_{\ell _{3t+1},\ell _{3t+2},\ell _{3t+3}}$ and
$C^{i',j'}_{\ell _{3t+1},\ell _{3t+2},\ell _{3t+3}}.$ 
Hence, each pair of color classes 
$C^{i,j}_{\ell _{3t+1},\ell _{3t+2},\ell _{3t+3}}$, $C^{i',j'}_{\ell _{3t+1},\ell _{3t+2},\ell _{3t+3}}$ intersects.
\end{itemize}
In consequence, the coloring is complete.
\end{proof}

To conclude this section we prove our third main theorem.

\begin{proof}[Proof of Theorem \ref{thm:tres}]
As $v=q^{2k},$ from Theorem \ref{teo-achro} we get
$$ 
\begin{aligned}
\left( \frac{q^k+1-\epsilon }{3}(q^k+2)+\epsilon \right)\frac{q^k-1}{q-1} 
& =\frac{q^{3k}+(2-\epsilon )q^{2k}+(2\epsilon -1)q^{k}-2-\epsilon }{3(q-1)} \\
 & =\frac{1}{3} \frac{\sqrt{v}(v-1)}{q-1}+\frac{(2-\epsilon )v+2\epsilon \sqrt{v}-2-\epsilon }{3(q-1)}, 
 \end{aligned}
 $$
which proves the statement.
\end{proof}


\section{Small dimensions}\label{section4}

In this section, we improve the previous bounds for dimensions two and three. 
First, we prove the exact values of achromatic and pseudoachromatic indices of finite affine planes. 
Due to the fact that there exist non-desarguesian affine planes, we use the notation $\mathrm{A}_q$ for an
arbitrary affine plane of order $q.$ For the axiomatic definition of $\mathrm{A}_q$ we refer to \cite{dem}. The
basic combinatorial properties of $\mathrm{A}_q$ are the same as of $\mathrm{AG}(2,q)$.

\begin{theorem}\label{thm:chrom_achrom_affine} 
Let $\mathrm{A}_q$ be any affine plane of order $q.$ Then
\[\chi'(\mathrm{A}_q)=\alpha'(\mathrm{A}_q)=q+1.\]
\end{theorem}

\begin{proof}
There are $q+1$ parallel classes of lines in $\mathrm{A}_q,$ let $\mathscr{S}_1,\mathscr{S}_2,\ldots,\mathscr{S}_{q+1}$ 
denote them. The lines in each class give a partition of the set of point of $\mathrm{A}_q$ and two lines have a point in common if and
only if they belong to distinct parallel classes. Hence, if we define a coloring $\phi $ with $q+1$ colors such that a line $\ell $ gets color
$i$ if and only if $\ell \in \mathscr{S}_i$ then $\phi $ is proper. This shows the inequality $q+1\leq \chi'(\mathrm{A}_q).$

Since $\chi'(\mathrm{A}_q)\leq\alpha'(\mathrm{A}_q),$ it is enough to prove that $\alpha'(\mathrm{A}_q)\leq q+1.$ 
Suppose to the contrary that $\alpha'(\mathrm{A}_q)\geq q+2,$ and let $\psi $ be a complete and proper coloring with 
more than $q+1$ color classes, say $C_1,C_2,\ldots,C_m.$ As $\psi $ is proper, each color class must be a subset of a parallel class. 
There are more color classes than parallel classes, hence, by the pigeonhole principle, there are at least two color classes that are
subsets of the same parallel class. If $C_i, C_j\subset \mathscr{S}_k$ then the elements of $C_i$ have empty intersection with 
the elements of $C_j$ contradicting to the completeness of $\psi.$ 
Thus $\alpha'(\mathrm{A}_q)\leq q+1,$ the theorem is proved.   
\end{proof}

\begin{theorem} 
\label{thm:upper}
Let $\mathrm{A}_q$ be any affine plane of order $q.$ Then
$$\psi' (\mathrm{A}_q)=\left\lfloor \tfrac{(q+1)^2}{2} \right\rfloor.$$
\end{theorem}

\begin{proof}
First, we prove that $\psi' (\mathrm{A}_q)\leq \left\lfloor \tfrac{(q+1)^2}{2} \right\rfloor.$
Suppose to the contrary that there exists a complete
coloring $\varphi $ of $\mathrm{A}_q$ with $\left\lfloor \tfrac{(q+1)^2}{2} \right\rfloor + 1$ color classes. 
As $\mathrm{A}_q$ has $q^2+q$ lines, this implies that
$\varphi $ has at most $q^2+q- \left( \lfloor \tfrac{(q+1)^2}{2} \rfloor + 1 \right) $ color classes of 
cardinality greater than one. Thus, there are at least 
$$\left\lfloor \tfrac{(q+1)^2}{2} \right\rfloor + 1 - \left( q^2+q- \left( \left\lfloor \tfrac{(q+1)^2}{2} \right\rfloor + 1 \right) \right)
=\begin{cases}
q + 2, & \text{if  } q \text{ is even},\\
q + 3, & \text{if } q \text{ is odd},
\end{cases}
$$ 
color classes of size one. 
Hence, again by the pigeonhole principle, there are at least two color classes of size one such that they belong to the
same parallel class. This means that they have empty intersection, so $\varphi $ is not complete. 
This contradiction shows that $\psi' (\mathrm{A}_q)\leq \left\lfloor \tfrac{(q+1)^2}{2} \right\rfloor.$

We continue to give a complete coloring of $\mathrm{A}_q$ with $\left\lfloor \tfrac{(q+1)^2}{2} \right\rfloor$ color classes. 
Let $P$ be a point and $e_1,e_2,\dots , e_{q+1}$ be the lines through $P$. For $i=1,2,,\dots ,q+1$ 
let $\mathscr{S}_i$ be the parallel class containing $e_i$ and denote the $q-1$ lines in the set  
$\mathscr{S}_i\setminus \{ e_i\} $ by $\ell _i, \ell _{(q+1)+i},\dots ,\ell _{(q-2)(q+1)+i}.$ Then
$$\bigcup _{i=1}^q \left( \mathscr{S}_i\setminus \{ e_i\} \right) =\{ \ell _1,\ell _2, \dots ,\ell _{q^2-1}\} ,$$
and the lines $\ell _j$ and $\ell _{j+1}$ belong to distinct parallel classes for all $1\leq j <q^2-1.$ 
For better clarity, we construct $q+1$ color classes with even indices and 
$\left\lfloor \tfrac{q^2-1}{2} \right\rfloor $ color classes with odd indices.
Let the color class $C_{2k}$ consist of one element, the line $e_k,$ for $k=1,2,\dots , q+1,$ 
let the color class $C_{2k-1}$ contain the lines $\ell _{2k-1}$ and $\ell _{2k}$ for $k=1,2,\dots ,\left\lfloor \tfrac{q^2-1}{2} \right\rfloor ,$
finally, if $q$ is even, let the color class $C_{q^2-3}$ contain the line $\ell _{q^2-1},$ too.

The coloring is complete, because color classes having even indices intersect at $P,$ and each color class with odd index contains
two non-parallel lines whose union intersects all lines of the plane.  
\end{proof}

\smallskip
Our last construction gives a lower bound for the achromatic index of $\mathrm{AG}(3,q)$.
As $\alpha'(\mathrm{AG}(3,q))\leq \psi '(\mathrm{AG}(3,q)),$ this can be considered as well as 
lower estimate on the pseudoachromatic index of 
$\mathrm{AG}(3,q)$ and this bound is better than the general one proved in Theorem \ref{odd-part}.
We use the cyclic model of $\mathrm{PG}(2,q)$ to make the coloring. The detailed description of this
model can be found in \cite[Theorem 4.8 and Corollary 4.9]{jwph3}. We collect the most important 
properties of the cyclic model in the following proposition.

\begin{proposition}
\label{cyclicrepr}
Let $q$ be a prime power. Then the group $\mathbb{Z}_{q^2+q+1}$ admits a perfect difference
set $D=\{ d_0,d_1,d_2,\dots , d_q\} ,$ that is the $q^2+q$ integers $d_i-d_j$ are all
distinct modulo $q^2+q+1.$ We may assume without loss of generality that $d_0=0$ and $d_1=1.$
The points and lines of the plane $\mathrm{PG}(2,q)$ can be represented in the following way.
The points are the elements of $\mathbb{Z}_{q^2+q+1},$ the lines are the subsets 
$$D+j=\{ d_i+j\, \colon \, d_i\in D\} $$
for $j=0,1,\dots ,q^2+q,$ and the incidence is the set-theoretical inclusion. 
\end{proposition}

\begin{theorem}
The achromatic index of $\mathrm{AG}(3,q)$ satisfies
the inequality:
$$\frac{q(q+1)^2}{2}+1  \leq\alpha'(\mathrm{AG}(3,q)).$$
\end{theorem}  

\begin{proof}
Consider the projective closure of the affine space, let
$\mathrm{PG}(3,q)=\mathrm{AG}(3,q)\cup {\mathcal H}_{\infty }.$
Then the parallel classes of affine lines correspond to the points of
${\mathcal H}_{\infty }.$ This plane is isomorphic to
$\mathrm{PG}(2,q),$ hence it has a cyclic representation (described in Proposition \ref{cyclicrepr}).
Let $v=q^2+q+1,$ let the points and the lines 
of ${\mathcal H}_{\infty }$ be $P_1,P_2,\dots ,P_v,$ and
$\ell _1,\ell _2,\dots ,\ell _v,$ respectively. We can choose the 
numbering such that for $i=1,2,3,\dots ,v$ the line $\ell _i$ contains the
points $P_i,$ $P_{i+1}$ and $P_{i-d}$ (where $0\neq d\neq 1$ is a fixed element 
of the difference set $D,$ and the subscripts are taken modulo $v$).

The affine parts of the planes of 
$\mathrm{PG}(3,q)$ intersecting ${\mathcal H}_{\infty }$ in a fixed line $\ell_i$ 
form a set of $q$ parallel planes in
$\mathrm{AG}(3,q).$ 
We denote this set by  
$A(P_i)=\{ \Pi _{P_i,1},\Pi _{P_i,2},\dots ,\Pi _{P_i,q}\} .$ 
Let $W_i$ be a plane of 
$\mathrm{PG}(3,q)$ intersecting ${\mathcal H}_{\infty }$ in $\ell _{i-d}.$
Then each element of $A(P_i)\cup A(P_{i+1})$ intersects $W_i$ in a line which 
passes on the point $P_i,$ so we can choose the
numbering of the elements of $A(P_i)$ and $A(P_{i+1})$, such that $\Pi _{P_i,j}\cap \Pi _{P_{i+1},j}\subset W_i$ 
for $i=1,3,\dots ,v-2$ and $j=1,2,\dots ,q.$ 
Let $e^i_j$ denote the line $\Pi _{P_i,j}\cap \Pi _{P_{i+1},j}.$

We assign $q+1$ color classes to the pair $(P_i,P_{i+1})$ for $i=1,3,\dots ,v-2.$
Let the color class $C^i_0$ contain 
the affine parts of the lines $e^i_1,e^i_2\dots ,e^i_q.$ 
For $j=1,2,\dots ,q$, let the color class $C^i_j$ contain the 
parallel lines of $\Pi_{P_i,j}$ passing on $P_i$ except the line $e^i_j,$ 
and the $q$ parallel lines of $\Pi_{P_{i+1},j}$ passing on $P_{i+1}$.    
Finally, let the color class $C^v$ contain the affine parts of all lines through $P_v.$ 
In this way we constructed
$$(q+1)\frac{v-1}{2}+1=\frac{q(q+1)^2}{2}+1 $$
color classes and each line belongs to exactly one of them, because
$C^i_0$ contains $q$ lines, $C^i_j$ contains 
$2q-1$ lines for each $j=1,2,\dots ,q.$ and $C^v$ contains $q^2$ lines.

The coloring is proper by definition. 
The color class $C^v$ obviously intersects any other class. 
For other pairs of color classes, two major cases are distinguished when we prove the completeness.  
On the one hand, if $i\neq k$ then we have:
\begin{itemize} 
\item
$C^i_0\cap C^k_0\neq  \emptyset ,$ because the planes $W_i$ and $W_k$ intersect each other,
\item
if $j>0$ then $C^i_0\cap C^k_j\neq  \emptyset ,$ because the planes $W_i$ and $\Pi_{P_{k+1},j}$ intersect each other,
\item
if $m>0$ and $j>0$ then $C^i_m\cap C^k_j\neq  \emptyset ,$ because the planes $\Pi_{P_{i+1},m}$ and $\Pi_{P_{k+1},j}$ intersect each other. 
\end{itemize}
On the other hand, color classes having the same superscript also have non-empty intersection:
\begin{itemize} 
\item
$C^i_0\cap C^i_j\neq  \emptyset,$ because the planes $W_i$ and $\Pi_{P_{i+1},j}$ intersect each other, 
\item
if $j\neq k$ then the planes $\Pi_{P_{i},j}$ and $\Pi_{P_{i+1},k}$
intersect in a line $f$ and $f\neq e^i_j,$ hence its points are not removed from $\Pi_{P_{i},j},$ so
$C^i_j\cap C^i_k\neq  \emptyset .$ 
\end{itemize}

Hence the coloring is also complete, this proves the theorem.
\end{proof}

\section*{Acknowledgments}

The authors gratefully acknowledge funding from the following sources:
Gabriela Araujo-Pardo was partially supported by CONACyT-M{\' e}xico under Projects 178395, 166306, 
and by PAPIIT-M{\' e}xico under Project IN104915. Gy{\" o}rgy Kiss was partially supported by the 
bilateral Slovenian-Hungarian Joint Research Project, grant no. NN 114614 (in Hungary) and N1-0032 (in Slovenia), 
and by the Hungarian National Foundation for Scientific Research, grant no. K 124950. 
Christian Rubio-Montiel was partially supported by a CONACyT-M{\' e}xico Postdoctoral fellowship, 
and by the National scholarship programme of the Slovak republic. 
Adri{\' a}n V{\' a}zquez-{\' A}vila was partially supported by SNI of CONACyT-M{\' e}xico.


\end{document}